\documentclass{amsart}

\usepackage{ adjustbox }
\usepackage{ amsthm }
\usepackage{ amssymb }
\usepackage{ amsmath }
\usepackage{ array }
\usepackage{ booktabs }
\usepackage{ cite }
\usepackage{ color }
\usepackage{ enumitem }
\usepackage{ latexsym }
\usepackage{ url }
\usepackage[all]{ xy }

\theoremstyle{definition}

\newtheorem{definition}{Definition}[section]

\newtheorem{conjecture}[definition]{Conjecture}
\newtheorem{example}[definition]{Example}

\theoremstyle{plain}

\newtheorem{lemma}[definition]{Lemma}
\newtheorem{proposition}[definition]{Proposition}
\newtheorem{theorem}[definition]{Theorem}

\makeatletter
\renewcommand{\tocsection}[3]{%
\indentlabel{\@ifnotempty{#2}{\makebox[1.50em][l]{\ignorespaces#1#2.}}}#3}
\renewcommand{\tocsubsection}[3]{%
\indentlabel{\@ifnotempty{#2}{\hspace*{1.50em}\makebox[2.25em][l]{\ignorespaces#1#2.}}}#3}
\renewcommand{\tocsubsubsection}[3]{%
\indentlabel{\@ifnotempty{#2}{\hspace*{3.75em}\makebox[3.00em][l]{\ignorespaces#1#2.}}}#3}
\makeatother

\allowdisplaybreaks

\begin{document}

\title[Algebraic identities for linear operators]
{Algebraic identities for linear operators on associative triple systems (long version)}

\author{Murray R. Bremner}

\address{Department of Mathematics and Statistics,
University of Saskatchewan,
Saskatoon, Canada}

\email{bremner@math.usask.ca}



\begin{abstract}
We present the first classification of algebraic identities in 3 variables 
for linear operators on associative structures.  
We work in the context of associative triple systems, but since any associative algebra 
with product $xy$ becomes an associative triple system with product $xyz$, 
our results apply to associative algebras as well.  
This is the first time that Rota's classification problem for linear operators 
has been extended to algebras with an $n$-ary operation for $n \ge 3$.  
Our work is an application of computational linear algebra to the classification problem 
for linear operators.
%Our methods are similar to those in earlier work of the author and Elgendy on operator identities 
%in 2 variables.  
We begin with a generic operator identity with indeterminate coefficients.  
From this we use operadic partial compositions to derive a large sparse matrix whose nonzero entries 
are the indeterminates.  
We follow the rank principle which states that significant operator identities correspond to coefficients 
which produce submaximal rank of the matrix.  
For operator identities of multiplicity 1 (each term contains the operator once)
we obtain 6 families with 1 parameter, and 1 isolated solution.
For multiplicity 2, we obtain 6 families with 2 parameters, 27 families with 1 parameter, 
and 9 isolated solutions.  
\end{abstract}

\maketitle

{\footnotesize\tableofcontents}

%%%%%%%%%%%%%%%%%%%%%%%%%%%%%%%%%%%%%%%%%%%%%%%%%%%%%%%%%%%%%%%%%%%%%%%%%%%%%%%%%%%%%%%%%%%%%%%%
%%%%%%%%%%%%%%%%%%%%%%%%%%%%%%%%%%%%%%%%%%%%%%%%%%%%%%%%%%%%%%%%%%%%%%%%%%%%%%%%%%%%%%%%%%%%%%%%
%%%%%%%%%%%%%%%%%%%%%%%%%%%%%%%%%%%%%%%%%%%%%%%%%%%%%%%%%%%%%%%%%%%%%%%%%%%%%%%%%%%%%%%%%%%%%%%%

\section{Introduction}

%%%%%%%%%%%%%%%%%%%%%%%%%%%%%%%%%%%%%%%%%%%%%%%%%%%%%%%%%%%%%%%%%%%%%%%%%%%%%%%%%%%%%%%%%%%%%%%%

\subsection{Background}

Gian-Carlo Rota \cite{Rota1995} introduced the somewhat ill-defined problem of classifying
the algebraic identities that can be satisfied by linear operators on associative algebras.
He said: 
``I have posed the problem of finding all possible algebraic identities that can be 
satisfied by a linear operator on an algebra.
Simple computations show that the possibilities are very few, and the problem of classifying
all such identities is very probably completely solvable.''
Some of the most familiar operator identities are the derivation identity, 
the Rota-Baxter identity, the averaging identities, and the Nijenhuis identity;
furthermore, these identities give rise in a natural way to important algebraic structures,
such as diassociative algebras, dendriform algebras, averaging algebras, and N-dendriform algebras;
see \cite{Aguiar,GaoZhang2018,Guo2012,LG2012,PeiGuo2015}.
See the author and Elgendy \cite{B2025, BE2022} for a list of well-known operator identities, 
and some identities which were previously unknown.

Unfortunately, Rota did not provide details of his computations, 
and subsequent work on this problem has revealed different approaches to a possible solution.
One approach is through the theory of Gr\"obner-Shirshov (GS) bases for operated algebras.
A first step in this direction was Guo's work on operated semigroups \cite{Guo2009}.
The work of Guo et al.~\cite{GSZ2011,GSZ2013} applied GS bases to the classification
of operator identities of differential type and of Rota-Baxter type.
The work of the author and Elgendy \cite{BE2022} introduced another approach in which
one starts with a generic operator identity with indeterminate coefficients,
and then uses operadic partial compositions to derive the matrix of consequences
(a large sparse matrix whose nonzero entries are the indeterminate coefficients).
That paper also introduced the rank principle, which states that significant identities correspond 
to coefficients which produce submaximal rank of the matrix of consequences.
It is this latter approach which we will follow in this paper.
  
Rota's problem may also be generalized in various ways, providing new operator identities 
that deserve further study.
In particular, this problem has been considered in the setting of Lie algebras; 
see Zhang et al.~\cite{ZGG2023,ZGWP2024}. 
This introduces the possibility of considering Rota's problem for 
not necessarily associative algebras.

Previous work on Rota's problem has been restricted to
identities in 2 variables for operators on algebras with a binary operation.
The present paper classifies algebraic identities in 3 variables 
for linear operators on associative triple systems.  
Since any associative algebra 
with product $xy$ becomes an associative triple system with product $xyz$, 
our results apply to associative algebras as well.  
This is the first time that Rota's classification problem has been considered 
for linear operators on algebraic structures with an $n$-ary operation for $n \ge 3$.  

%%%%%%%%%%%%%%%%%%%%%%%%%%%%%%%%%%%%%%%%%%%%%%%%%%%%%%%%%%%%%%%%%%%%%%%%%%%%%%%%%%%%%%%%%%%%%%%%

\subsection{Structure of this paper}

In \S\ref{opmon} we introduce the notion of operator identity in the general context
of an associative $n$-ary algebra for $n \ge 2$.
Any operator identity is a linear combination of operator monomials.
We provide an algorithm for generating all operator monomials of a given degree
(number of indeterminates in each term) and given multiplicity (number of occurrences 
of the operator in each term).
Recent work of Au and the author \cite{AB2025} on generalized Narayana numbers
provides an exact formula for the number of operator monomials, 
depending on the arity of the associative operation,
the weight (the total number of unary operators and $n$-ary multiplications),
and the multiplicity of the operator.

In \S\ref{pcmc} we recall the notion of partial composition in an algebraic operad,
and explain the construction of the matrix of consequences for an operator identity.

In \S\ref{mult1} we consider operator identities on associative triple systems
which are linear combinations of the 4 ternary operator monomials of weight 2 and multiplicity 1.
In this case, the rank principle produces only a few significant identities,
including the ternary derivation identity.

In \S\ref{mult2} we consider operator identities on associative triple systems
which are linear combinations of the 10 ternary operator monomials of weight 3 and multiplicity 2.
In this case, the rank principle produces a complex set of operator identities:
9 isolated solutions, 27 families with 1 parameter, and 6 families with 2 parameters.  

The linear algebra calculations described in this paper were performed with the
computer algebra system Maple.

%%%%%%%%%%%%%%%%%%%%%%%%%%%%%%%%%%%%%%%%%%%%%%%%%%%%%%%%%%%%%%%%%%%%%%%%%%%%%%%%%%%%%%%%%%%%%%%%
%%%%%%%%%%%%%%%%%%%%%%%%%%%%%%%%%%%%%%%%%%%%%%%%%%%%%%%%%%%%%%%%%%%%%%%%%%%%%%%%%%%%%%%%%%%%%%%%
%%%%%%%%%%%%%%%%%%%%%%%%%%%%%%%%%%%%%%%%%%%%%%%%%%%%%%%%%%%%%%%%%%%%%%%%%%%%%%%%%%%%%%%%%%%%%%%%

\section{Operator monomials}
\label{opmon}

Although the main results in this paper refer to associative triple systems (also known as
ternary associative algebras), with an operation of arity 3, in this section we begin by 
considering the case of general arity $n \ge 2$.

For background on $n$-ary associative algebras, see Carlsson
\cite{Carlsson} which adopts the point of view of classical algebraic structure theory, 
and the work of Gnedbaye \cite{Gnedbaye1,Gnedbaye2} which adopts the more modern point 
of view based on algebraic operads.
For general background on algebraic operads, see Markl et al.~\cite{MSS-book} for applications,
the comprehensive theoretical treatise of Loday and Vallette \cite{LV-book}, 
and the author and Dotsenko \cite{BD-book} which emphasizes the algorithmic aspects.

\begin{definition}
An \emph{associative $n$-ary algebra} (or \emph{associative $n$-tuple system})
is a vector space $A$ over a field $\mathbb{F}$ together with a multilinear product
$p\colon A^n \to A$ satisfying the \emph{$n$-ary associativity identities}
for $0 \le i < j \le n{-}1$:
\begin{align*}
&
p( a_1, \dots, a_i, p( a_{i+1}, \dots, a_{i+n} ), a_{i+n+1}, \dots, a_{2n-1} )
=
\\
&
p( a_1, \dots, a_j, p( a_{j+1}, \dots, a_{j+n} ), a_{j+n+1}, \dots, a_{2n-1} ).
\end{align*}
\end{definition}

Any monomial in an associative $n$-ary algebra has a degree
(number of indeterminates) which is congruent to 1 modulo $n{-}1$.
A simple inductive argument shows that the associativity identities imply
that the value of any monomial is independent of the placement of the operation symbols.

\begin{definition}
Let $A$ be an associative $n$-ary algebra and let $L\colon A \to A$ be a linear operator on 
the underlying vector space of $A$.
By an \emph{$n$-ary operator monomial} we mean a monomial (product of elements of $A$)
into which have been inserted any number of operator symbols,
respecting the $n$-ary nature of the operation in $A$.
The number of indeterminates in the monomial is called its \emph{degree} $d$.
The number of operator symbols in the monomial is called its \emph{multiplicity} $m$.
The total number of operations in the monomial (including both unary and $n$-ary)
is called its \emph{weight}.
\end{definition}

Alternatively, we may define operator monomials recursively as follows:
\begin{itemize}
\item
The single indeterminate $x$ is an operator monomial.
\item
If $M_1, \dots, M_n$ are operator monomials, then so is $p( M_1, \dots, M_n )$.
\item
If $M$ is an operator monomial, then so is $L(M)$.
\end{itemize}
We also impose the condition that in any operator monomial of degree $d$, 
we enumerate the indeterminates $x_1, \dots, x_d$ from left to right,
using the identity permutation of the indices, so that the indices indicate
merely the position of the indeterminate in the monomial.
(In other words, the underlying operad is nonsymmetric.)

To impose a total order on operator monomials, we recall the following concept.

\begin{definition}
A \emph{Dyck word} is a string of parentheses 
--
that is, a word from the alphabet $\{ \, (, \, ) \, \}$
--
which is \emph{balanced} in the sense that
\begin{itemize}
\item
the word contains the same number of left and right parentheses
(thus the length of every Dyck word is even), and
\item
in every initial sequence (starting from the left), 
the number of left parentheses is greater than or equal to the number of right parentheses.
\end{itemize}
A subsequence of the form $(\,)$ in a Dyck word is called a \emph{nesting}.
We define a total order called the \emph{lex order} on Dyck words as follows.
Given two words $v$ and $w$ of the same length, let $i$ be the least index for which $v_i \ne w_i$.
Then we write $v \prec w$ and say that $v$ \emph{precedes} $w$
if and only if $v_i$ is $($ and $w_i$ is $)$.
\end{definition}

\begin{definition}
We recursively define a map $r$ from operator monomials $M$ of degree $d$ and multiplicity $m$
to Dyck words $w$ of length $2(d{+}m)$:
\begin{itemize}
\item
if $M$ is an indeterminate then $r(M) = (\,)$,
\item
if $M_1, \dots, M_n$ are ($n$-ary) operator monomials then 
\[
r( M_1 \cdots M_n ) = r(M_1) \cdots r(M_n),
\]
\item
if $M$ is an operator monomial then $r( L(M) ) = ( r(M) )$.
\end{itemize}
In other words, if $M$ is an operator monomial then $r(M)$ is obtained by
replacing every indeterminate by a nesting and removing every occurrence
of the operator symbol $L$.
It is easy to see that the map $r$ is bijective for arity $n = 2$ but merely
injective for arity $n \ge 3$.
In both cases, we obtain a lex order on operator monomials by defining
$M_1 \prec M_2$ if and only if $r( M_1 ) \prec r( M_2 )$.
\end{definition}

To clarify these definitions, Table \ref{mult123} displays the ternary operator
monomials of degree $d = 3$ (one associative product) and multiplicities $m = 1, 2, 3$ in lex order 
together with the corresponding Dyck words.
We use the generic indeterminate symbol $\ast$, which is justified
by the fact that the index of an indeterminate merely
indicates its position in the monomial.
To make the correspondence between monomials and Dyck words easier to see,
we have not used standard power notation for iterated operators;
thus we write $L(L(\ast))$ not $L^2(\ast)$.

\begin{table}
\begin{align*}
&
\text{multiplicity 1}
\\
&
\begin{array}{rllrll}
  1  &\quad  L({*}{*}{*})  &\quad  (()()())  &\quad
  2  &\quad  L({*}){*}{*}  &\quad  (())()()  \\ 
  3  &\quad  {*}L({*}){*}  &\quad  ()(())()  &\quad
  4  &\quad  {*}{*}L({*})  &\quad  ()()(())  
\end{array}
\\
&
\text{multiplicity 2}
\\
&
\begin{array}{rllrll}
  1  &\quad  L(L({*}{*}{*}))  &\quad  ((()()()))  &\quad 
  2  &\quad  L(L({*}){*}{*})  &\quad  ((())()())  \\ 
  3  &\quad  L(L({*})){*}{*}  &\quad  ((()))()()  &\quad
  4  &\quad  L({*}L({*}){*})  &\quad  (()(())())  \\ 
  5  &\quad  L({*}{*}L({*}))  &\quad  (()()(()))  &\quad
  6  &\quad  L({*})L({*}){*}  &\quad  (())(())()  \\ 
  7  &\quad  L({*}){*}L({*})  &\quad  (())()(())  &\quad
  8  &\quad  {*}L(L({*})){*}  &\quad  ()((()))()  \\ 
  9  &\quad  {*}L({*})L({*})  &\quad  ()(())(())  &\quad
 10  &\quad  {*}{*}L(L({*}))  &\quad  ()()((()))  
\end{array}
\\
&
\text{multiplicity 3}
\\
&
\begin{array}{rllrll}
  1  &\quad  L(L(L({*}{*}{*})))  &\quad  (((()()())))  &\quad
  2  &\quad  L(L(L({*}){*}{*}))  &\quad  (((())()()))  \\ 
  3  &\quad  L(L(L({*})){*}{*})  &\quad  (((()))()())  &\quad
  4  &\quad  L(L(L({*}))){*}{*}  &\quad  (((())))()()  \\ 
  5  &\quad  L(L({*}L({*}){*}))  &\quad  ((()(())()))  &\quad
  6  &\quad  L(L({*}{*}L({*})))  &\quad  ((()()(())))  \\ 
  7  &\quad  L(L({*})L({*}){*})  &\quad  ((())(())())  &\quad
  8  &\quad  L(L({*}){*}L({*}))  &\quad  ((())()(()))  \\ 
  9  &\quad  L(L({*}))L({*}){*}  &\quad  ((()))(())()  &\quad
 10  &\quad  L(L({*})){*}L({*})  &\quad  ((()))()(())  \\ 
 11  &\quad  L({*}L(L({*})){*})  &\quad  (()((()))())  &\quad
 12  &\quad  L({*}L({*})L({*}))  &\quad  (()(())(()))  \\ 
 13  &\quad  L({*}{*}L(L({*})))  &\quad  (()()((())))  &\quad
 14  &\quad  L({*})L(L({*})){*}  &\quad  (())((()))()  \\ 
 15  &\quad  L({*})L({*})L({*})  &\quad  (())(())(())  &\quad
 16  &\quad  L({*}){*}L(L({*}))  &\quad  (())()((()))  \\ 
 17  &\quad  {*}L(L(L({*}))){*}  &\quad  ()(((())))()  &\quad
 18  &\quad  {*}L(L({*}))L({*})  &\quad  ()((()))(())  \\ 
 19  &\quad  {*}L({*})L(L({*}))  &\quad  ()(())((()))  &\quad
 20  &\quad  {*}{*}L(L(L({*})))  &\quad  ()()(((())))  \\[-10pt]
\end{array}
\end{align*}
\caption{Ternary operator monomials of degree 3 and multiplicities 1, 2, 3
in lex order together with corresponding Dyck words}
\label{mult123}
\end{table}

See Figure \ref{monomialalgorithm} for an algorithm in pseudocode (based on Maple)
which generates all $n$-ary operator monomials up to a given degree and multiplicity.
In this algorithm, the symbol $X$ is used as a generic argument symbol, corresponding
to $*$ in the operator monomials and $(\,)$ in the Dyck words.
Each monomial is a list (an ordered set enclosed in brackets) and brackets are also
used without ambiguity to indicate the application of the operator (as parentheses
are used in Dyck words).
For example, we list below the operator monomials and the original 
$X$-lists generated by the algorithm for degree 3 and multiplicity 2;
the order is that produced by the algorithm, which is not the same as that in
Table \ref{mult123}:
\[
\begin{array}{rllrll}
  1  &\quad  [[[X, X, X]]]  &\quad  L(L({*}{*}{*}))  &\quad\quad
  2  &\quad  [[X, X, [X]]]  &\quad  L({*}{*}L({*}))  \\ 
  3  &\quad  [[X, [X], X]]  &\quad  L({*}L({*}){*})  &\quad\quad
  4  &\quad  [[[X], X, X]]  &\quad  L(L({*}){*}{*})  \\ 
  5  &\quad  [X, X, [[X]]]  &\quad  {*}{*}L(L({*}))  &\quad\quad
  6  &\quad  [X, [X], [X]]  &\quad  {*}L({*})L({*})  \\ 
  7  &\quad  [X, [[X]], X]  &\quad  {*}L(L({*})){*}  &\quad\quad
  8  &\quad  [[X], X, [X]]  &\quad  L({*}){*}L({*})  \\ 
  9  &\quad  [[X], [X], X]  &\quad  L({*})L({*}){*}  &\quad\quad
 10  &\quad  [[[X]], X, X]  &\quad  L(L({*})){*}{*}  \\ 
\end{array}
\]
To justify the algorithm, observe that any monomial of multiplicity $q \ge 1$
can be obtained from a monomial of multiplicity $q{-}1$ by enclosing some valid
sublist in brackets, where valid means that the sublist has an appropriate length
to be an $n$-ary associative monomial.

\begin{figure}
\bigskip
\begin{itemize}
\item
$n$ := the arity of the associative operation
\item
\texttt{maxp} := maximum degree of operator monomials to be generated
\item
\texttt{maxq} := maximum multiplicity of operator monomials to be generated
\item
for $p$ from 1 to \texttt{maxp} in steps of $n{-}1$ do
($p$ is the degree of the monomials)
\begin{enumerate}
\item
(multiplicity 0: only monomial is list of $p$ argument symbols $X$)
\item[]
\texttt{monomials}[ $p$, 0 ] := [ [ $X$, \dots, $X$ ] ]
\item
for $q$ to \texttt{maxq} do ($q$ is the multiplicity of the monomials)
\begin{enumerate}
\item
\texttt{monomials}$[p,q]$ := $[\;]$ \; (the empty list)
\item
(loop through all monomials with one less operator symbol)
\item[]
for $m$ in \texttt{monomials}$[ p, q{-}1 ]$ do
\begin{enumerate}
\item
$k$ := length($m$) (the number of items in the list $m$)
\item
(double loop through all submonomials of $m$)
\item[]
for $i$ to $k$ do for $j$ from $i$ to $k$ do
\begin{itemize}
\item
if $j{-}i{+}1 \equiv 1$ modulo $n{-}1$ then
\begin{itemize}
\item
(add operator brackets around valid submonomial)
\item[]
$m'$ := [ $m_1$, \dots, $m_{i-1}$, [ $m_i$, \dots, $m_j$ ], $m_{j+1}$, \dots, $m_k$ ]
\item
append $m'$ to the list \texttt{monomials}$[p,q]$
\end{itemize}
\end{itemize}
\end{enumerate}
\item
eliminate any repetitions from the list \texttt{monomials}$[p,q]$
\end{enumerate}
\end{enumerate}
\end{itemize}
\vspace{-10pt}
\caption{Algorithm for generating $n$-ary operator monomials}
\label{monomialalgorithm}
\end{figure}

\begin{definition}
An \emph{operator polynomial} $R$ (or \emph{operator identity} $R \equiv 0$) 
of degree $d$ and multiplicity $m$ for an associative $n$-ary operation
is any nonzero linear combination of the corresponding operator monomials
with coefficients in $\mathbb{F}$.
Two operator polynomials are \emph{equivalent} if they differ only by a nonzero
scalar multiple.
\end{definition}

To conclude this section, we recall the formula for the number of $n$-ary operator monomials 
of given weight and multiplicity which has been obtained by Au and the author \cite{AB2025},
which provides a generalization of the familiar Narayana numbers.

\begin{proposition}
Let $N_n(w,m)$ be the number of operator monomials of weight $w$ and multiplicity $m$
for an associative $n$-ary operation.
Then
\[
N_n(w,m) =
\frac{1}{w+1}
\binom{w+1}{m+1}
\binom{w+(w-m)(n-2)+1}{m}
\]
\end{proposition}

\begin{example}
\label{narayanaexample}
For $n = 3$ we obtain the following table:
\[
\begin{array}{l|rrrrrr}
w \backslash m & 0 & 1 & 2 & 3 & 4 & 5 \\
\midrule
0 & 1 \\
1 & 1 & 1 \\
2 & 1 & 4 & 1 \\
3 & 1 & 9 & 10 & 1 \\
4 & 1 & 16 & 42 & 20 & 1 \\
5 & 1 & 25 & 120 & 140 & 35 & 1
\end{array}
\]
The subdiagonal gives the tetrahedral numbers 1, 4, 10, 20, \dots
(compare Table \ref{mult123}).
\end{example}

%%%%%%%%%%%%%%%%%%%%%%%%%%%%%%%%%%%%%%%%%%%%%%%%%%%%%%%%%%%%%%%%%%%%%%%%%%%%%%%%%%%%%%%%%%%%%%%%
%%%%%%%%%%%%%%%%%%%%%%%%%%%%%%%%%%%%%%%%%%%%%%%%%%%%%%%%%%%%%%%%%%%%%%%%%%%%%%%%%%%%%%%%%%%%%%%%
%%%%%%%%%%%%%%%%%%%%%%%%%%%%%%%%%%%%%%%%%%%%%%%%%%%%%%%%%%%%%%%%%%%%%%%%%%%%%%%%%%%%%%%%%%%%%%%%

\section{The matrix of consequences}
\label{pcmc}

Our primary object of study in this paper is the matrix of consequences of a generic operator identity.
To construct this matrix, we need to recall the concept of partial composition from 
the theory of algebraic operads.
We consider only nonsymmetric operads, so that every monomial
of degree $d$ has indeterminates $x_1 \cdots x_d$ with the subscripts given
by the identity permutation.
Equivalently, we may use the generic argument symbol $*$ throughout.

Let $\alpha$ and $\beta$ be operations of arities $m$ and $n$ respectively.
Then the $i$-th \emph{partial composition} of $\beta$ into $\alpha$, 
denoted $\alpha \circ_i \beta$ where $1 \le i \le m$,
is the operation of arity $m+n-1$ obtained by replacing argument $i$ of $\alpha$ 
by an instance of operation $\beta$:
\[
\alpha \circ_i \beta
=
\alpha( x_1, \dots, x_{i-1}, \beta( x_i, \dots, x_{i+n-1} ), x_{i+n}, \dots, x_{m+n-1} ).
\]
Partial composition is bilinear, and hence extends to linear
combinations of operations and to operator identities.

We consider only two operations, the ternary associative multiplication $p$
and the unary operator $L$.
Any operator monomial or identity can be considered as an operation in its own right.
Given an operator identity $R$ of degree $d$ and multiplicity $m$, we may form 
the following partial compositions:
\[
R \circ_i p \;\; (1 \le i \le d),
\qquad
p \circ_i R \;\; (1 \le i \le 3),
\qquad
R \circ_i L \;\; (1 \le i \le d),
\qquad
L \circ R.
\]
In the last case, the partial composition symbol has no subscript since there is only
one possibility.
The first two cases increase the degree and the last two cases increase the multiplicity.
For example, if $R = L^2({\ast}{\ast}{\ast})$ then
\begin{align*}
&
R \circ_i p = L^2({\ast}{\ast}{\ast}{\ast}{\ast}) 
\quad 
\text{for $1 \le i \le 3$ by ternary associativity},
\\
&
p \circ_1 R = L^2({\ast}{\ast}{\ast}){\ast}{\ast}, \qquad
p \circ_2 R = {\ast}L^2({\ast}{\ast}{\ast}){\ast}, \qquad
p \circ_3 R = {\ast}{\ast}L^2({\ast}{\ast}{\ast}),
\\
&
R \circ_1 L = L^2(L({\ast}){\ast}{\ast}), \qquad
R \circ_2 L = L^2({\ast}L({\ast}){\ast}), \qquad
R \circ_3 L = L^2({\ast}{\ast}L({\ast})),
\\
&
L \circ R = L^3({\ast}{\ast}{\ast}).
\end{align*}
Partial composition with $p$ increases the degree by 2 
but does not increase the multiplicity,
whereas partial composition with $L$ increases the multiplicity by 1 
but does not increase the degree.

To understand how an operator identity $R$ interacts with the two operations $p$ and $L$,
we need to apply partial composition with both operations, either with $p$ and then $L$,
or with $L$ and then $p$.
In both cases, the degree will increase by 2 and multiplicity by 1.
This gives the following possibilities, where the ranges on the subscripts are evident
and are not specified; some of these may be equal:
\begin{equation}
\label{pclist}
\left\{ \quad
\begin{array}{llll}
( R \circ_i p ) \circ_j L, &\quad
L \circ ( R \circ_i p ), &\quad
( p \circ_i R ) \circ_j L, &\quad
L \circ ( p \circ_i R ),
\\[6pt]
( R \circ_i L ) \circ_j p, &\quad
p \circ_j ( R \circ_i L ), &\quad
( L \circ R ) \circ_i p, &\quad
p \circ_i ( L \circ R ).
\end{array}
\right.
\end{equation}
We may now state the most important definition in this paper.

\begin{definition}
Let $R$ be an operator identity of degree $d$ and multiplicity $m$.
The operator identities of degree $d+2$ and multiplicity $m+1$
in display \eqref{pclist} are called the \emph{consequences} of $R$.
Let $k$ denote the number of distinct consequences of $R$,
and let $S_1, \dots, S_k$ be these consequences
(their order is not significant).
Let $m_1, \dots, m_\ell$ be the lex-ordered list of the operator monomials of
degree $d+2$ and multiplicity $m+1$.
The \emph{matrix of consequences} of $R$ is the $k \times \ell$ matrix in which
the $(i,j)$-entry is the coefficient of $m_j$ in $S_i$
for $1 \le i \le k$ and $1 \le j \le \ell$.
\end{definition}

%%%%%%%%%%%%%%%%%%%%%%%%%%%%%%%%%%%%%%%%%%%%%%%%%%%%%%%%%%%%%%%%%%%%%%%%%%%%%%%%%%%%%%%%%%%%%%%%
%%%%%%%%%%%%%%%%%%%%%%%%%%%%%%%%%%%%%%%%%%%%%%%%%%%%%%%%%%%%%%%%%%%%%%%%%%%%%%%%%%%%%%%%%%%%%%%%
%%%%%%%%%%%%%%%%%%%%%%%%%%%%%%%%%%%%%%%%%%%%%%%%%%%%%%%%%%%%%%%%%%%%%%%%%%%%%%%%%%%%%%%%%%%%%%%%

\section{Operator identities of degree 3 and multiplicity 1}
\label{mult1}

Previous work by the author and Elgendy \cite{BE2022} used determinantal ideals and Gr\"obner bases 
to classify all operator identities for a binary operation of degree 2 and multiplicities 1 and 2.  
In the present case of a ternary operation, computational experiments showed that those methods, 
relying heavily on commutative algebra, were not practical.  
However, all the operator identities from that previous paper have coefficients in the set 
$\{ 0, \pm 1, \pm 2 \}$, 
excluding the two 1-parameter families.
Therefore in the present paper we make this assumption on the coefficients and 
perform a direct search (using a Cartesian product loop) over all operator identities 
with those coefficients.  We may assume without loss of generality that 
the leading coefficient is positive (1 or 2), and
that the coefficients are relatively prime.

In this section we consider multiplicity 1, for which the computations are sufficiently small 
%(the matrix of consequences and the set of identities with submaximal rank) 
that we can present all the details.  
In the next section we consider multiplicity 2 where the some of the computations are too large 
to present explicitly.

For degree 3 and multiplicity 1, the four operator monomials are displayed in Table \ref{mult123}.  
Let $R$ be the generic linear combination of these operator monomials:
\[
R =
a \, L({*}{*}{*}) +
b \, L({*}){*}{*} +
c \, {*}L({*}){*} +
d \, {*}{*}L({*}),
\]
with indeterminate coefficients $a, b, c, d$.
The consequences of $R$ have degree 5 and multiplicity 2 and hence are linear combinations 
of the monomials in Table \ref{deg5mult2}.

\begin{table}
$
\begin{array}{rlrlrl}
  1  &\quad  L(L({*}{*}{*}{*}{*}))  &\qquad  
  2  &\quad  L(L({*}{*}{*}){*}{*})  &\qquad  
  3  &\quad  L(L({*}{*}{*})){*}{*}  \\
  4  &\quad  L(L({*}){*}{*}{*}{*})  &\qquad  
  5  &\quad  L(L({*}){*}{*}){*}{*}  &\qquad  
  6  &\quad  L(L({*})){*}{*}{*}{*}  \\ 
  7  &\quad  L({*}L({*}{*}{*}){*})  &\qquad  
  8  &\quad  L({*}L({*}){*}{*}{*})  &\qquad   
  9  &\quad  L({*}L({*}){*}){*}{*}  \\
 10  &\quad  L({*}{*}L({*}{*}{*}))  &\qquad   
 11  &\quad  L({*}{*}L({*}){*}{*})  &\qquad  
 12  &\quad  L({*}{*}L({*})){*}{*}  \\ 
 13  &\quad  L({*}{*}{*}L({*}){*})  &\qquad  
 14  &\quad  L({*}{*}{*}{*}L({*}))  &\qquad   
 15  &\quad  L({*}{*}{*})L({*}){*}  \\
 16  &\quad  L({*}{*}{*}){*}L({*})  &\qquad  
 17  &\quad  L({*})L({*}{*}{*}){*}  &\qquad  
 18  &\quad  L({*})L({*}){*}{*}{*}  \\ 
 19  &\quad  L({*}){*}L({*}{*}{*})  &\qquad  
 20  &\quad  L({*}){*}L({*}){*}{*}  &\qquad   
 21  &\quad  L({*}){*}{*}L({*}){*}  \\
 22  &\quad  L({*}){*}{*}{*}L({*})  &\qquad   
 23  &\quad  {*}L(L({*}{*}{*})){*}  &\qquad  
 24  &\quad  {*}L(L({*}){*}{*}){*}  \\ 
 25  &\quad  {*}L(L({*})){*}{*}{*}  &\qquad  
 26  &\quad  {*}L({*}L({*}){*}){*}  &\qquad   
 27  &\quad  {*}L({*}{*}L({*})){*}  \\
 28  &\quad  {*}L({*}{*}{*})L({*})  &\qquad   
 29  &\quad  {*}L({*})L({*}{*}{*})  &\qquad  
 30  &\quad  {*}L({*})L({*}){*}{*}  \\ 
 31  &\quad  {*}L({*}){*}L({*}){*}  &\qquad  
 32  &\quad  {*}L({*}){*}{*}L({*})  &\qquad   
 33  &\quad  {*}{*}L(L({*}{*}{*}))  \\
 34  &\quad  {*}{*}L(L({*}){*}{*})  &\qquad   
 35  &\quad  {*}{*}L(L({*})){*}{*}  &\qquad  
 36  &\quad  {*}{*}L({*}L({*}){*})  \\ 
 37  &\quad  {*}{*}L({*}{*}L({*}))  &\qquad  
 38  &\quad  {*}{*}L({*})L({*}){*}  &\qquad   
 39  &\quad  {*}{*}L({*}){*}L({*})  \\
 40  &\quad  {*}{*}{*}L(L({*})){*}  &\qquad   
 41  &\quad  {*}{*}{*}L({*})L({*})  &\qquad  
 42  &\quad  {*}{*}{*}{*}L(L({*}))  
\end{array}
$
\bigskip
\caption{Ternary operator monomials of degree 5, multiplicity 2}
\label{deg5mult2}
\end{table}

Straightforward computations show that there are $k = 42$ distinct consequences of $R$.
Coincidentally, there are also $\ell = 42$ operator monomials of degree 5 and multiplicity 2.
The resulting $42 \times 42$ matrix of consequences is displayed in Figure \ref{matrix31}.
The columns correspond in lex order to the monomials in Table \ref{deg5mult2}.  
The consequences are calculated using partial compositions as described in the previous section.  
For example, the first row of the matrix of consequences corresponds to
\[
( R \circ_1 p ) \circ_1 L
=
a \, L(L({*}){*}{*}{*}{*}) +
b \, L(L({*}){*}{*}){*}{*} +
c \, L({*}){*}{*}L({*}){*} +
d \, L({*}){*}{*}{*}L({*}).
\]
These monomials have indices 4, 5, 21, 22 in Table \ref{deg5mult2}, and these are the columns
in which the indeterminates appear in Figure \ref{matrix31}.

\begin{figure}[h]
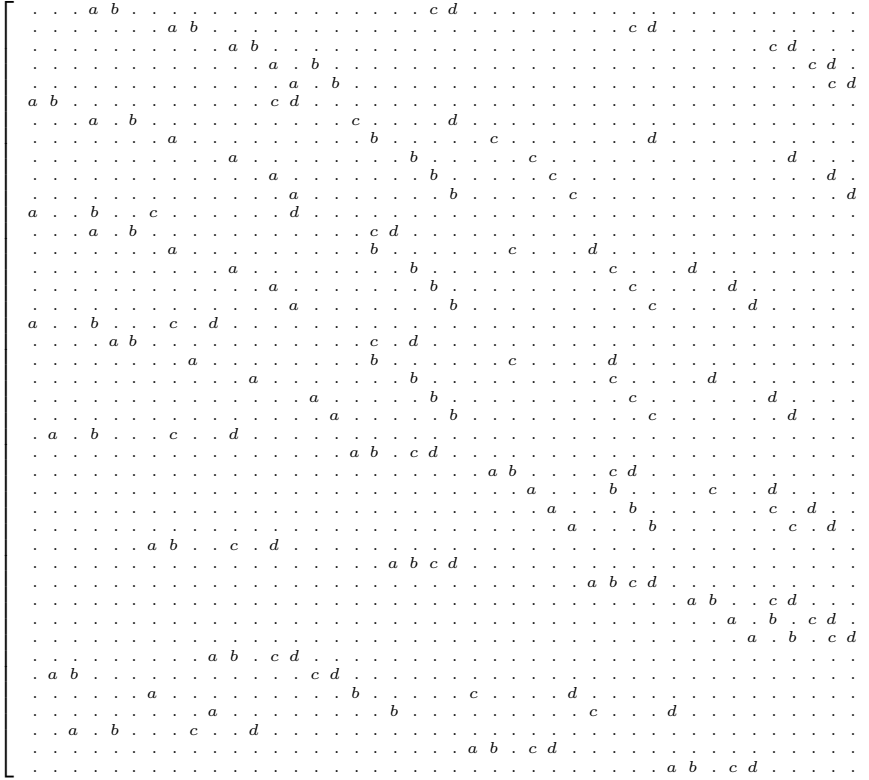

\[
\tiny
\left[
\begin{array}{rrrrrrrrrrrrrrrrrrrrrrrrrrrrrrrrrrrrrrrrrr}
. &\!\!\!\!\! . &\!\!\!\!\! . &\!\!\!\!\! a &\!\!\!\!\! b &\!\!\!\!\! . &\!\!\!\!\! . &\!\!\!\!\! . &\!\!\!\!\! . &\!\!\!\!\! . &\!\!\!\!\! . &\!\!\!\!\! . &\!\!\!\!\! . &\!\!\!\!\! . &\!\!\!\!\! . &\!\!\!\!\! . &\!\!\!\!\! . &\!\!\!\!\! . &\!\!\!\!\! . &\!\!\!\!\! . &\!\!\!\!\! c &\!\!\!\!\! d &\!\!\!\!\! . &\!\!\!\!\! . &\!\!\!\!\! . &\!\!\!\!\! . &\!\!\!\!\! . &\!\!\!\!\! . &\!\!\!\!\! . &\!\!\!\!\! . &\!\!\!\!\! . &\!\!\!\!\! . &\!\!\!\!\! . &\!\!\!\!\! . &\!\!\!\!\! . &\!\!\!\!\! . &\!\!\!\!\! . &\!\!\!\!\! . &\!\!\!\!\! . &\!\!\!\!\! . &\!\!\!\!\! . &\!\!\!\!\! . \\
. &\!\!\!\!\! . &\!\!\!\!\! . &\!\!\!\!\! . &\!\!\!\!\! . &\!\!\!\!\! . &\!\!\!\!\! . &\!\!\!\!\! a &\!\!\!\!\! b &\!\!\!\!\! . &\!\!\!\!\! . &\!\!\!\!\! . &\!\!\!\!\! . &\!\!\!\!\! . &\!\!\!\!\! . &\!\!\!\!\! . &\!\!\!\!\! . &\!\!\!\!\! . &\!\!\!\!\! . &\!\!\!\!\! . &\!\!\!\!\! . &\!\!\!\!\! . &\!\!\!\!\! . &\!\!\!\!\! . &\!\!\!\!\! . &\!\!\!\!\! . &\!\!\!\!\! . &\!\!\!\!\! . &\!\!\!\!\! . &\!\!\!\!\! . &\!\!\!\!\! c &\!\!\!\!\! d &\!\!\!\!\! . &\!\!\!\!\! . &\!\!\!\!\! . &\!\!\!\!\! . &\!\!\!\!\! . &\!\!\!\!\! . &\!\!\!\!\! . &\!\!\!\!\! . &\!\!\!\!\! . &\!\!\!\!\! . \\
. &\!\!\!\!\! . &\!\!\!\!\! . &\!\!\!\!\! . &\!\!\!\!\! . &\!\!\!\!\! . &\!\!\!\!\! . &\!\!\!\!\! . &\!\!\!\!\! . &\!\!\!\!\! . &\!\!\!\!\! a &\!\!\!\!\! b &\!\!\!\!\! . &\!\!\!\!\! . &\!\!\!\!\! . &\!\!\!\!\! . &\!\!\!\!\! . &\!\!\!\!\! . &\!\!\!\!\! . &\!\!\!\!\! . &\!\!\!\!\! . &\!\!\!\!\! . &\!\!\!\!\! . &\!\!\!\!\! . &\!\!\!\!\! . &\!\!\!\!\! . &\!\!\!\!\! . &\!\!\!\!\! . &\!\!\!\!\! . &\!\!\!\!\! . &\!\!\!\!\! . &\!\!\!\!\! . &\!\!\!\!\! . &\!\!\!\!\! . &\!\!\!\!\! . &\!\!\!\!\! . &\!\!\!\!\! . &\!\!\!\!\! c &\!\!\!\!\! d &\!\!\!\!\! . &\!\!\!\!\! . &\!\!\!\!\! . \\
\end{array}
\right]
\]
\vspace{-10pt}
\caption{Matrix of consequences for the generic operator identity of degree 3 and multiplicity 1
(with dot for zero)}
\label{matrix31}
\end{figure}

\begin{table}
$
\begin{array}{rrrrrr|rrrrrr}
\# &\qquad a & b & c & d & \text{rank} &
\# &\qquad a & b & c & d & \text{rank} \\ \midrule
 1  &\qquad   1 &  0 &  0 &  0  &\qquad  27  &\qquad
 2  &\qquad   1 &  0 &  0 &  1  &\qquad  32  \\ 
 3  &\qquad   1 &  0 &  0 & -1  &\qquad  27  &
 4  &\qquad   1 &  0 &  0 &  2  &\qquad  32  \\ 
 5  &\qquad   1 &  0 &  0 & -2  &\qquad  32  &
 6  &\qquad   1 &  1 &  0 &  0  &\qquad  32  \\ 
 7  &\qquad   1 & -1 &  0 &  0  &\qquad  27  &
 8  &\qquad   1 & -1 &  0 & -1  &\qquad  32  \\ 
 9  &\qquad   1 & -1 &  1 & -1  &\qquad  27  &
10  &\qquad   1 & -1 & -1 & -1  &\qquad  27  \\ 
11  &\qquad   1 & -1 &  2 & -1  &\qquad  32  &
12  &\qquad   1 & -1 & -2 & -1  &\qquad  32  \\ 
13  &\qquad   1 &  2 &  0 &  0  &\qquad  32  &
14  &\qquad   1 & -2 &  0 &  0  &\qquad  32  \\
\end{array}
$
\bigskip
\caption{Submaximal rank solutions for degree 3 and multiplicity 1}
\label{submax1}
\end{table}

We calculate that the matrix of consequences has rank 36 over the function field $\mathbb{F}(a,b,c,d)$.  
So our goal is to find all values of the indeterminates $a,b,c,d$ which produce a matrix of rank $\le 35$.  Altogether there are 312 values of the coefficients in the Cartesian product loop, 
but only 41 ($\approx 13.14\%$) of these produce submaximal rank.  
These solutions are displayed in Table \ref{submax1}.
The set of possible submaximal ranks is $\{ 26, 27, 31, 32, 34 \}$.
Inspection of Table \ref{submax1} reveals six 1-parameter families
and one isolated solution:
\begin{enumerate}
\item
Solutions 1--5, 15, 16 combine to form the family $[ 1, 0, 0, d ]$ with generic rank 32
and special rank 27 for $d = 0, -1$.
\item
Solutions 1, 6, 7, 13, 14, 17, 18 form the 1-parameter family $[ 1, b, 0, 0 ]$ with generic rank 32 
and special rank 27 for $b = 0, -1$.
\item
Solutions 8--12, 19, 20 form the family $[ 1, -1, c, -1 ]$ with generic rank 32
and special rank 27 for $c = 1, -1$.
Note that $c = -1$ corresponds to the ternary derivation identity,
\[
L({\ast}{\ast}{\ast})
=
L({\ast}){\ast}{\ast} +
{\ast}L({\ast}){\ast} +
{\ast}{\ast}L({\ast}).
\]
\item
Solutions 21--25, 30, 31 form the family $[0,1,0,d]$ with generic rank 34
and special ranks 26 for $d = 0$ and 31 for $d = -1$.
\item
Solutions 21, 26--29, 32, 33 form the family $[ 0, 1, c, 0 ]$ with generic rank 31
and special rank 26 for $c = 0$.
\item
Solutions 34--40 form the family $[ 0, 0, 1, d ]$ with generic rank 31 and special 
rank 26 for $d = 0$.
\item
Solution 41, namely $[0,0,0,1]$ with rank 26, is an isolated point.
\end{enumerate}
These computations were done over the field $\mathbb{Q}$ of rational numbers.
The generic ranks of the 1-parameter families were verified by computing
the rank of the matrix of consequences over the appropriate field of rational functions.
Thus, for example, for family (3) we set $a = 1$, $b = -1$, $c$ free, $d = -1$ 
in the matrix of consequences and compute its rank over $\mathbb{Q}(c)$.
Summarizing, we have the following result.

\begin{theorem}
\label{mult1theorem}
Let $R$ be an operator identity of degree 3 and multiplicity 1 on an associative triple system.  
Assume that the coefficient vector $v$ of $R$ in $\mathbb{F}^4$ satisfies these conditions:
\begin{itemize}
\item
the components of $v$ belong to the set $\{0,\pm 1, \pm 2 \}$, and
\item
$v$ produces submaximal rank in the matrix of consequences.
\end{itemize}
Then $R$ is equivalent to one of the following identities:
\begin{align*}
&
L({\ast}{\ast}{\ast}) + d \, {\ast}{\ast}L({\ast}) \equiv 0,
\qquad
L({\ast}{\ast}{\ast}) + b \, L({\ast}){\ast}{\ast} \equiv 0,
\\
&
L({\ast}{\ast}{\ast}) - L({\ast}){\ast}{\ast} + c \, {\ast}L({\ast}){\ast} - {\ast}{\ast}L({\ast}) \equiv 0,
\\
&
L({\ast}){\ast}{\ast} + d \, {\ast}{\ast}L({\ast}) \equiv 0,
\qquad
L({\ast}){\ast}{\ast} + c \, {\ast}L({\ast}){\ast} \equiv 0,
\\
&
{\ast}L({\ast}){\ast} + d \, {\ast}{\ast}L({\ast}) \equiv 0,
\qquad
{\ast}{\ast}L({\ast}) \equiv 0.
\end{align*}
\end{theorem}

\begin{conjecture}
Let $R$ be an operator identity of degree 3 and multiplicity 1 on an associative triple system.  
Assume that the coefficient vector $v$ of $R$ in $\mathbb{F}^4$ produces submaximal rank in 
the matrix of consequences (with no restrictions on the components of $v$).
Then $R$ is equivalent to one of the identities of Theorem \ref{mult1theorem}.
\end{conjecture}

%%%%%%%%%%%%%%%%%%%%%%%%%%%%%%%%%%%%%%%%%%%%%%%%%%%%%%%%%%%%%%%%%%%%%%%%%%%%%%%%%%%%%%%%%%%%%%%%
%%%%%%%%%%%%%%%%%%%%%%%%%%%%%%%%%%%%%%%%%%%%%%%%%%%%%%%%%%%%%%%%%%%%%%%%%%%%%%%%%%%%%%%%%%%%%%%%
%%%%%%%%%%%%%%%%%%%%%%%%%%%%%%%%%%%%%%%%%%%%%%%%%%%%%%%%%%%%%%%%%%%%%%%%%%%%%%%%%%%%%%%%%%%%%%%%

\section{Operator identities of degree 3 and multiplicity 2}
\label{mult2}

In this section we consider ternary operator identities of degree 3 and multiplicity 2.
The generic such identity is the linear combination with indeterminate coefficients
of the 10 monomials of multiplicity 2 in Table \ref{mult123}:
\begin{align*}
R &=
a_1    L^2({*}{*}{*}) +     
a_2    L(L({*}){*}{*}) +   
a_3    L^2({*}){*}{*} +    
a_4    L({*}L({*}){*}) +   
a_5    L({*}{*}L({*})) 
\\
&+  
a_6    L({*})L({*}){*} +  
a_7    L({*}){*}L({*}) +    
a_8    {*}L^2({*}){*} +   
a_9    {*}L({*})L({*}) +    
a_{10} {*}{*}L^2({*}).    
\end{align*}
As in the case of multiplicity 1, we consider only coefficients from the set 
$\{ 0, \pm 1, \pm 2 \}$, and assume without loss of generality that 
the leading coefficient is positive and that the coefficients are relatively prime.
In the case of multiplicity 2, the Cartesian product loop over all 10 coefficients
must consider 4882812 cases.
Remarkably, we will see that only 387 ($\approx 0.0079\%$) of these coefficient vectors 
produce submaximal rank of the matrix of consequences.
This is strong evidence that the rank principle is a powerful tool for isolating 
significant operator identities.

The consequences of $R$ have degree 5 and multiplicity 3;
according to Example \ref{narayanaexample} there are 140 such monomials
and hence they are not displayed.
As in the case of multiplicity 1, there are 42 distinct consequences of the identity $R$.
(The same sequences of partial compositions provide a non-redundant set of consequences.)
Therefore the matrix of consequences for multiplicity 3 has size $42 \times 140$ and hence
is not displayed; as before the columns correspond in lex order to the operator monomials of degree 5
and multiplicity 3 and the rows correspond to the consequences in any convenient order.
For example, the first row of the matrix of consequences corresponds to
\begin{align*}
&
( R \circ_1 p ) \circ_1 L
=
a_1    L^2(L({*}){*}{*}{*}{*}) +     
a_2    L(L(L({*}){*}{*}){*}{*}) +   
a_3    L^2(L({*}){*}{*}){*}{*} 
\\
&
+    
a_4    L(L({*}){*}{*}L({*}){*}) +   
a_5    L(L({*}){*}{*}{*}L({*})) +  
a_6    L(L({*}){*}{*})L({*}){*} 
\\
&
+ 
a_7    L(L({*}){*}{*}){*}L({*}) +    
a_8    L({*}){*}{*}L^2({*}){*} +   
a_9    L({*}){*}{*}L({*})L({*}) +    
a_{10} L({*}){*}{*}{*}L^2({*}).    
\end{align*}
We calculate that the matrix of consequences has full rank 42 over the function field 
$\mathbb{Q}(a_1,\dots,a_{10})$.  
So our goal is to find all values of the indeterminates $a_1,\dots,a_{10}$ which produce 
rank $\le 41$ of the matrix of consequences.  
As already mentioned, there are only 387 such coefficient vectors;
these are listed in the Appendix.

These initial computations were done over the field with $p = 1009$ elements in order to 
save computing time.
To justify this we use the following fact.

\begin{lemma}
Let $A$ be an integer matrix and let $p$ be a prime number.
Let $r$ be the rank of $A$ over $\mathbb{Q}$ and let $s$ be the rank of $A$ over $\mathbb{F}_p$.
Then $s \le r$.
\end{lemma}

\begin{proof}
The rank is the maximal size of a nonzero minor.%
%\footnote{I thank Vladimir Dotsenko for this observation.}
\end{proof}

It follows that if the modular computations show that a given coefficient vector produces
maximal rank, then the rational rank must also be maximal.
However, there is a small chance that the modular rank is submaximal but the rational rank
is maximal.
We need therefore to calculate the ranks of the 387 solutions using rational arithmetic.
We found that the modular and rational ranks agree in every case.
The set of possible submaximal ranks is $\{ 30, 34, 35, 36, 40, 41 \}$.

Inspection of the Appendix leads to the conclusion that the 387 solutions 
with submaximal rank may be classified into 
9 isolated points, 
27 families with 1 parameter, and
6 families with 2 parameters.
These solutions are given in Table \ref{allsolutions},
together with the corresponding ranks of the matrix of consequences.
For the families with parameters, these are the generic ranks 
over the corresponding field of rational functions;
for some special values of the parameters, the rank will be lower (see the Appendix).
Summarizing, we have the following result.

\begin{table}
$
\begin{array}{rrrrrrrrrr|r}
a_1 & a_2 & a_3 & a_4 & a_5 & a_6 & a_7 & a_8 & a_9 & a_{10} & \text{rank} \\
\midrule
\multicolumn{10}{l|}{\text{Isolated points}} &   \\
  1 & -2 &  1 &  2 & -2 & -2 &  2 &  1 & -2 &  1  &  40 \\ 
  1 & -2 &  1 & -2 & -2 &  2 &  2 &  1 &  2 &  1  &  40 \\ 
  1 & -2 &  1 &  0 & -2 &  0 &  2 &  0 &  0 &  1  &  41 \\ 
  0 &  1 &  0 &  0 & -1 &  0 &  0 &  0 &  0 &  0  &  41 \\ 
  0 &  1 &  0 &  0 &  1 &  0 & -1 &  0 &  0 &  0  &  41 \\ 
  0 &  0 &  1 &  0 &  0 &  0 &  1 &  0 &  0 &  1  &  41 \\ 
  0 &  0 &  0 &  0 &  0 &  1 &  1 &  1 &  1 &  0  &  40 \\ 
  0 &  0 &  0 &  0 &  0 &  1 & -1 & -1 &  1 &  0  &  40 \\ 
  0 &  0 &  0 &  0 &  0 &  0 &  0 &  0 &  0 &  1  &  30 \\ 
\midrule
\multicolumn{10}{l|}{\text{One-parameter families}} &   \\
  1 & a_2 &  0 &  0 & {-}1 &   0 & {-}a_2 &  0 &   0 &     0  & 41  \\ %  7 
  1 & {-}1 &  0 &  0 & a_5 &   0 & {-}a_5 &  0 &   0 &     0  & 41  \\ % 10 
  1 &  0 & {-}1 &  0 &  0 &   0 &   0 & a_8 &   0 &    {-}1  & 36  \\ %  5 
  1 &  0 &  0 & a_4 &  0 &   0 &   0 &  0 &   0 &     0  & 41  \\ %  2 
  0 &  1 & a_3 &  0 &  0 &   0 &   0 &  0 &   0 &     0  & 36  \\ % 24 
  0 &  1 & {-}1 & a_4 &  0 & {-}a_4 &   0 &  0 &   0 &     0  & 40  \\ % 26 
  0 &  1 & {-}1 &  0 & a_5 &   0 & {-}a_5 &  0 &   0 &     0  & 41  \\ % 25 
  0 &  1 &  0 & a_4 &  0 & {-}a_4 &   0 &  0 &   0 &     0  &  41 \\ % 23 
  0 &  1 &  0 & a_4 &  0 &   0 &   0 &  0 &   0 &     0  &  41 \\ % 21 
  0 &  1 &  0 & a_4 &  0 &   0 &  {-}1 &  0 & {-}a_4 &     0  & 41  \\ % 22 
  0 &  1 &  0 &  0 &  0 &   0 &  a_7 &  0 &   0 &     0  & 41  \\ % 19 
  0 &  1 &  0 &  0 & a_5 &   0 &  {-}1 &  0 &   0 &   {-}a_5  & 41  \\ % 20 
  0 &  0 &  1 &  0 &  0 &   0 &  a_7 &  0 &   0 &     0  &  41 \\ % 27 
  0 &  0 &  1 &  0 &  0 &   0 &  a_7 &  0 &   0 & {-}a_7{-}1  & 40  \\ % 28 
  0 &  0 &  0 &  1 & a_5 &  {-}1 & {-}a_5 &  0 &   0 &     0  & 41  \\ % 37 
  0 &  0 &  0 &  1 & a_5 &   0 &   0 &  0 &   0 &     0  & 41  \\ % 34 
  0 &  0 &  0 &  1 & a_5 &   0 &   0 &  0 &  {-}1 &     0  & 41  \\ % 35 
  0 &  0 &  0 &  1 & a_5 &   0 &   0 &  0 &  {-}1 &   {-}a_5  &  40 \\ % 36 
  0 &  0 &  0 &  1 &  0 &  a_6 &   0 &  0 &   0 &     0  &  41 \\ % 32 
  0 &  0 &  0 &  1 &  0 &  {-}1 &   0 & a_8 &  {-}1 &     0  & 41  \\ % 33 
  0 &  0 &  0 &  1 &  0 &   0 &   0 &  0 &  a_9 &     0  &  41 \\ % 30
  0 &  0 &  0 &  1 &  0 &   0 &  a_7 &  0 &   0 &     0  & 41  \\ % 31 
  0 &  0 &  0 &  0 &  1 &   0 &  a_7 &  0 &   0 &     0  & 41  \\ % 39 
  0 &  0 &  0 &  0 &  1 &   0 &   0 &  0 &   0 &   a_{10}  & 36  \\ % 38 
  0 &  0 &  0 &  0 &  0 &   1 &   0 & a_8 &   0 &     0  &  36 \\ % 40 
  0 &  0 &  0 &  0 &  0 &   0 &   1 &  0 &   0 &   a_{10}  & 41  \\ % 41 
  0 &  0 &  0 &  0 &  0 &   0 &   0 &  0 &   1 &   a_{10}  & 36  \\ % 43 
\midrule
\multicolumn{10}{l|}{\text{Two-parameter families}} &   \\
  1 & a_2 &    a_3 &  0 &  0 &   0 &   0 &  0 &   0 &     0  & 36  \\ %  4 {-} 2 pa_ra_meters
  1 & a_2 & {-}a_2{-}1 & a_4 &  0 & {-}a_4 &   0 &  0 &   0 &     0  & 41  \\ %  9 {-} 2 pa_ra_meters 
  1 &  0 &     0 & a_4 & a_5 &   0 &   0 &  0 & {-}a_4 & {-}a_5{-}1  & 41  \\ %  3 {-} 2 pa_ra_meters
  1 &  0 &     0 &  0 & a_5 &   0 &   0 &  0 &   0 &   a_{10}  &  36 \\ %  1 {-} 2 pa_ra_meters
  0 &  0 &     1 &  0 &  0 &  a_6 &   0 & a_8 &   0 &     0  & 36  \\ % 29 {-} 2 pa_ra_meters
  0 &  0 &     0 &  0 &  0 &   0 &   0 &  1 &  a_9 &   a_{10} & 36  \\    % 42 {-} 2 pa_ra_meters
\midrule
\end{array}
$
\medskip
\caption{Submaximal rank solutions for degree 3 and multiplicity 2}
\label{allsolutions}
\end{table}

\begin{theorem}
\label{mult2theorem}
Let $R$ be an operator identity of degree 3 and multiplicity 2 on an associative triple system.  
Assume that the coefficient vector $v$ of $R$ in $\mathbb{F}^{10}$ satisfies these conditions:
\begin{itemize}
\item
the components of $v$ belong to the set $\{0,\pm 1, \pm 2 \}$, and
\item
$v$ produces submaximal rank in the matrix of consequences.
\end{itemize}
Then $R$ is equivalent to one of the operator identities 
whose coefficient vectors appear in Table \ref{allsolutions}.
\end{theorem}

The operator identities corresponding to the first 3 isolated points are of particular interest.
We display them here explicitly:
\begin{align*}
&
    L^2({\ast}{\ast}{\ast})  
-2  L(L({\ast}){\ast}{\ast}) 
+   L^2({\ast}){\ast}{\ast}  
+2  L({\ast}L({\ast}){\ast}) 
-2  L({\ast}{\ast}L({\ast}))
\\
&\quad 
-2  L({\ast})L({\ast}){\ast} 
+2  L({\ast}){\ast}L({\ast}) 
+   {\ast}L^2({\ast}){\ast}  
-2  {\ast}L({\ast})L({\ast}) 
+   {\ast}{\ast}L^2({\ast})  
\equiv 0,
\\
&
    L^2({\ast}{\ast}{\ast})  
-2  L(L({\ast}){\ast}{\ast}) 
+   L^2({\ast}){\ast}{\ast}  
-2  L({\ast}L({\ast}){\ast}) 
-2  L({\ast}{\ast}L({\ast})) 
\\
&\quad 
+2  L({\ast})L({\ast}){\ast} 
+2  L({\ast}){\ast}L({\ast}) 
+   {\ast}L^2({\ast}){\ast}  
+2  {\ast}L({\ast})L({\ast}) 
+   {\ast}{\ast}L^2({\ast})  
\equiv 0,
\\
&
    L^2({\ast}{\ast}{\ast})  
-2  L(L({\ast}){\ast}{\ast}) 
+   L^2({\ast}){\ast}{\ast}  
-2  L({\ast}{\ast}L({\ast})) 
+2  L({\ast}){\ast}L({\ast}) 
+   {\ast}{\ast}L^2({\ast})  
\equiv 0.
\end{align*}

\begin{conjecture}
Let $R$ be an operator identity of degree 3 and multiplicity 2 on an associative triple system.  
Assume that the coefficient vector $v$ of $R$ in $\mathbb{F}^{10}$ produces submaximal rank in 
the matrix of consequences (with no restrictions on the components of $v$).
Then $R$ is equivalent to one of the identities of Theorem \ref{mult2theorem}.
\end{conjecture}

%%%%%%%%%%%%%%%%%%%%%%%%%%%%%%%%%%%%%%%%%%%%%%%%%%%%%%%%%%%%%%%%%%%%%%%%%%%%%%%%%%%%%%%%%%%%%%%%
%%%%%%%%%%%%%%%%%%%%%%%%%%%%%%%%%%%%%%%%%%%%%%%%%%%%%%%%%%%%%%%%%%%%%%%%%%%%%%%%%%%%%%%%%%%%%%%%
%%%%%%%%%%%%%%%%%%%%%%%%%%%%%%%%%%%%%%%%%%%%%%%%%%%%%%%%%%%%%%%%%%%%%%%%%%%%%%%%%%%%%%%%%%%%%%%%

\section{Appendix}
\label{mult2app}

\noindent
The 387 submaximal rank solutions for degree 3 and multiplicity 2:


%%%%%%%%%%%%%%%%%%%%%%%%%%%%%%%%%%%%%%%%%%%%%%%%%%%%%%%%%%%%%%%%%%%%%%%%%%%%%%%%%%%%%%%%%%%%%%%%
%%%%%%%%%%%%%%%%%%%%%%%%%%%%%%%%%%%%%%%%%%%%%%%%%%%%%%%%%%%%%%%%%%%%%%%%%%%%%%%%%%%%%%%%%%%%%%%%
%%%%%%%%%%%%%%%%%%%%%%%%%%%%%%%%%%%%%%%%%%%%%%%%%%%%%%%%%%%%%%%%%%%%%%%%%%%%%%%%%%%%%%%%%%%%%%%%

%%%%%%%%%%%%%%%%%%%%%%%%%%%%%%%%%%%%%%%%%%%%%%%%%%%%%%%%%%%%%%%%%%%%%%%%%%%%%%%%%%%%%%%%%%%%%%%%
%%%%%%%%%%%%%%%%%%%%%%%%%%%%%%%%%%%%%%%%%%%%%%%%%%%%%%%%%%%%%%%%%%%%%%%%%%%%%%%%%%%%%%%%%%%%%%%%
%%%%%%%%%%%%%%%%%%%%%%%%%%%%%%%%%%%%%%%%%%%%%%%%%%%%%%%%%%%%%%%%%%%%%%%%%%%%%%%%%%%%%%%%%%%%%%%%

\end{document}